\documentclass[12pt,twoside]{amsart}
\usepackage{amsmath}
\usepackage{amssymb}
\usepackage{amscd}
\usepackage{xypic}
\xyoption{all}
\title[On Serre Intersection Multiplicity Conjecture]{On Serre Intersection Multiplicity Conjecture}

\author{Mohammad Reza Rahmati}
\thanks{}
\address{ABDUS SALAM SCHOOL OF MATHEMATICAL SCIENCES, Pakistan
\hfill\break 
\hfill\break \\
\hfill\break }
\email{mrahmati@cimat.mx, rahmati@sms.edu.pk}

\newcommand{\comments}[1]{}

\newtheorem{theorem}{Theorem}[section]




\begin{document}

\begin{abstract}
In this short note, we expose some of the works on Serre intersection multiplicity conjecture. I provide a proof of the vanishing of Serre intersection multiplicity in non-proper intersection over a regular ring based on the intersection theory in W. Fulton. 
\end{abstract}

\maketitle

\vspace{0.5cm}

\section*{Introduction}

\vspace{0.5cm}

A definition of intersection multiplicity is one of the major tasks of any intersection theory. In 1950's J. P. Serre \cite{S} made a definition of Intersection multiplicity of two finitely generated modules over a regular local ring $A$ by a type of Euler characteristic, namely Tor-formula. Specifically, for $M,N$ two finite $A$-modules he defines their intersection multiplicity  

\vspace{0.2cm}

\begin{center}
$\chi^A(M,N) := \displaystyle{\sum_{i=0}^{\dim(A)} (-1)^i l(\text{Tor}_i^A(M,N)})$
\end{center}

\vspace{0.2cm}

\noindent
where $l$ denotes the length of the $A$-module, subject to the condition that $l(M \otimes_A N) < \infty$. The significance of this formula is the necessity of higher Tor's when considering complicated examples. The relation 

\vspace{0.2cm}

\begin{center}
$\dim(M) + \dim(N) \leq \dim(A)$ 
\end{center}

\vspace{0.2cm}

\noindent
will hold, \cite{S}. This may also be proved using the Auslander-Buchsbaum relation for modules over a regular ring $(A,\mathfrak{m})$, 

\vspace{0.2cm}

\begin{center}
$\text{pd}(M) + \text{depth}(M) = \dim(A)$ 
\end{center}

\vspace{0.2cm}

\noindent
and the formula $\dim(N) \leq \text{pd}(M)$ when $l(M \otimes_A N) <\infty$ known as Intersection Theorem, \cite{RO3}. In case of proper intersection, Serre definition agrees with the Hilbert-Samuel definition for multiplicity, that if $Y=\text{Spec}(A/\mathfrak{p})$ and $Z=\text{Spec}(A/\mathfrak{q})$ be subvarieties in $X=\text{Spec}(A)$ that intersect properly (the dimension condition reads as $\text{ht}(\mathfrak{p}) + \text{ht}(\mathfrak{q}) =\dim(A)$). Then it holds that $\sqrt{\mathfrak{p} + \mathfrak{q}} =\mathfrak{m}$, and we have 

\vspace{0.2cm}

\begin{center}
$\chi^A(A/\mathfrak{p},A/\mathfrak{q})= e_{\dim\mathfrak{q}}(\mathfrak{q},A/\mathfrak{p})$
\end{center}

\vspace{0.2cm}

\noindent
where $e_{\dim\mathfrak{q}}(\mathfrak{q},A/\mathfrak{p})$ is the Hilbert-Samuel multiplicity of the $A$-module $A/\mathfrak{p}$ with respect to the ideal $\mathfrak{q}$. More generally, let $A$ be noetherian and $\mathfrak{a}=<x_1,...,x_k>$, then 

\vspace{0.2cm}

\begin{center}
$\chi(K_{\bullet} \otimes M)=e_k(\mathfrak{a},M)$
\end{center}

\vspace{0.2cm}

\noindent
where $K_{\bullet}$ is the Koszul complex on $<x_1,...,x_k>$.

\vspace{0.5cm}

\noindent
\textbf{Conjecture:} \cite{S} Assume $A$ is a regular local ring and $M,N$ finite $A$-modules with $l(M \otimes_A N) <\infty$, then 

\vspace{0.2cm}

\begin{itemize}
\item[(1)] If $\dim M + \dim N < \dim A$, then $\chi^A(M,N) =0$\\[0.05cm]
\item[(2)] In case $\dim M + \dim N = \dim A$, called proper intersection, $\chi^A(M,N)>0$.
\end{itemize}

\vspace{0.2cm}

Serre proves most of the fundamental properties of the intersection multiplicity over regular rings, especially when they are essentially of finite type over a field or a discrete valuation ring
using the method of reduction to the diagonal, \cite{S}. Serre's proof essentially uses the flat structure over the field $k$. When $A=k[[x_1,...,x_n]]$, he regards $N$ as a module over a similar ring but the variables are changed with $y_1,...,y_n$ formally and writes 

\begin{equation}
M \otimes N \cong (M \hat{\otimes}_k N) \otimes_{k[[x_1,...,x_n,y_1,...,y_n]]}k[[x_1,..., y_n]]/I
\end{equation}

\vspace{0.2cm}

\noindent
where $y_1,...,y_n$ are some new variables and $I=(x_1-y_1,...,x_n -y_n)$. $\hat{\otimes}$ is the completed tensor product. Then, $\dim(M \hat{\otimes} N) \leq \dim(M) + \dim(N)$. Serre transforms the multiplicity question on $M,N$ to that of $(M \hat{\otimes}_k N)$ and the diagonal $\mathfrak{d}=<x_i \hat{\otimes} 1 -1 \hat{\otimes} x_i>$, by resolving the diagonal over $A \otimes_{A \otimes A} A$. We have the convergent spectral sequence 

\begin{equation}
E_{pq}^2=Tor_p^{A \hat{\otimes}_k A}((A \hat{\otimes}_k A)/\mathfrak{d},\widehat{Tor_q^k}(M,N)) \Longrightarrow Tor_{p+q}^A(M,N)
\end{equation}

\vspace{0.2cm}

\noindent
It follows that

\begin{center}
$\chi^A(M,N)=\sum (-1)^p Tor_p^{A \hat{\otimes}_k A}((A \hat{\otimes}_k A)/\mathfrak{d},\widehat{Tor_q^k}(M,N))$\\[0.3cm]
$\qquad \qquad =\sum_p (-1)^p l(H_p(\mathfrak{d},M \hat{\otimes}_k N))=e_{\mathfrak{d}}(M \hat{\otimes}_k N) \geq 0$
\end{center}

\vspace{0.2cm}

\noindent
Here $H_p(\mathfrak{d},M \hat{\otimes}_k N)$ is the Koszul cohomology with respect to the parameter system $\{x_i \hat{\otimes} 1 -1 \hat{\otimes} x_i \}$, and $e$ concerns Hilbert-Samuel multiplicity. The work of Serre in \cite{S} was actually enough for most of the purposes in geometry, specially in complex geometry, \cite{S}, \cite{SK}.  

\vspace{0.2cm}

An important step in understanding the definition of intersection multiplicity as an Euler characteristic is based on the Riemann-Roch theorem and theory of Chern characters, \cite{FL}, \cite{F}. P. Roberts  \cite{RO3} uses a theory of local chern characters for commutative rings. His proof concerns a theory of local chern characters with an application of Auslander-Buchsbaum theorem mentioned above. Specifically we have the Riemann-Roch formula in the form

\vspace{0.2cm}

\begin{center}
$\chi(F_{\bullet})=\text{ch}(F_{\bullet}).\tau(A)$ 
\end{center}

\vspace{0.2cm}

\noindent
for a regular local ring $A$, where $\text{ch}(F_{\bullet})=\text{ch}_0(F_{\bullet})+\text{ch}_1(F_{\bullet})+...$ with $\text{ch}_i(G_{\bullet}):\text{CH}_k(\text{Spec} A)_{\mathbb{Q}} \to A_{k-i}(\text{Supp} (G_{\bullet}))_{\mathbb{Q}}$ is the local chern character on the $K$-theory of perfect $A$-complexes, and $\tau$ is the Todd genus (Riemann-Roch homomorphism on regular schemes). We only need the formal properties of this theory to explain the vanishing part of the conjecture (see \cite{RO3} for details). 

\vspace{0.2cm}

\noindent
\textbf{Proof by P. Roberts:} \cite{RO1}, \cite{RO2}, \cite{RO3}, \cite{RO4} Let $M,N$ be as in the conjecture and $F_{\bullet}, G_{\bullet}$ be their free resolutions respectively. By the local Riemann-Roch 

\begin{equation}
\chi( F_{\bullet} \otimes G_{\bullet})=ch(F_{\bullet} \otimes G_{\bullet})[A]=\sum_{i+j=d} ch_i(F_{\bullet}). ch_j ( G_{\bullet})[A]
\end{equation}

\vspace{0.1cm}

Let $X$ and $Y$ denote the support of $M,N$ respectively, these are also the support of $F_{\bullet}, G_{\bullet}$. If $d-i > \dim(Y)$ then $ch_j(G_{\bullet})[A]=0$, and similarly when $d-j > \dim(X)$ then $ch_i(F_{\bullet})[A]=0$. These are the only cases when $\dim(X)+\dim(Y) <d$.

\vspace{0.2cm}

P. Roberts studies the relation of Serre's conjecture with other theorems and conjectures in commutative algebra, see \cite{RO4}. 

\vspace{0.2cm}

The definition of Intersection multiplicity lifts to $K$-groups. This means that $K$-theory of the ring $A$ hides an intersection theory. Gillet and Soul\'e \cite{GS1} systematically define intersection theory as a theory on $K_0$. The exterior powers endow a $\lambda$-ring structure on $K_0$ with support of a regular ring, 

\vspace{0.2cm}

\begin{center}
$\lambda^k:\displaystyle{\bigoplus_{Y \subset X} K_0^Y(X) \to \bigoplus_{Y \subset X} K_0^Y(X), \qquad \lambda^k([E^{\bullet}])=[\Lambda^k E^{\bullet}], \ k \geq 0}$
\end{center}

\vspace{0.2cm}

\noindent
(not homomorphisms) where $Y$ runs over closed subsets of $X$ and employ the associated Adams operations 

\vspace{0.2cm}

\begin{center}
$\psi^k: K_0^Y(X) \to K_0^Y(X), \qquad k \geq 0$
\end{center}

\vspace{0.2cm}

\noindent
(group homomorphisms). They analyze the eigen-values of Adams operations on the graded parts of $K_0$ with support with respect to the filtration by the codimension of support. 

\vspace{0.3cm}

\noindent
\textbf{The proof by H. Gillet and C. Soul\'e:} \cite{GS1} If $F^m$ is the filtration by co-dimension of support, i.e.

\begin{equation}
F^mK_0^Y(X):=\displaystyle{\lim_{\stackrel{Z \subset X}{\text{codim}_X(Z) \geq m}}Im (K_0^Z(X) \to K_0^Y(X))}
\end{equation}

\vspace{0.1cm}

\noindent
The vector space $K_0^Y(X)$ decomposes as $\displaystyle{\bigoplus_{i \geq codim(Y)}}Gr^iK_0^Y(X)_{\mathbb{Q}}$, and  

\[ Gr_F^i \psi_k:Gr_F^iK_0^Y(X)_{\mathbb{Q}} \to Gr_F^iK_0^Y(X)_{\mathbb{Q}} \] 

\vspace{0.2cm}

\noindent
is just multiplication by $k^i$, \cite{GS1}. We have the product

\begin{equation}
F^mK_0^Y(X)_{\mathbb{Q}} \otimes F^nK_0^Y(X)_{\mathbb{Q}} \to F^{m+n}K_0^Y(X)_{\mathbb{Q}}
\end{equation}

\vspace{0.2cm}

\noindent
If $\alpha=\sum \alpha_i, \ \beta=\sum \beta_i$, then $\alpha \cup \beta=\sum \alpha_i \cup \beta_j$. One checks that $\psi_k(\alpha_i \cup \beta_j)=k^{i+j}\alpha_i \cup \beta_j$ and thus 

\[ \alpha_i \cup \beta_j \in Gr_{i+j}K_0^{Y \cap Z}(X) \]

\vspace{0.2cm}

\noindent
It follows that when 

\[ \text{codim}(\text{Supp}(M))+\text{codim}(\text{Supp}(N)) > d=\dim A \] 

\vspace{0.2cm}

\noindent
then 

\begin{equation}
\chi(M,N)=([M] \cup [N]) \cap [\mathcal{O}_X]= \displaystyle{\sum (-1)^i l(Tor_i^A(M,N)}) =0
\end{equation}

\vspace{0.2cm}

\noindent
In this way the intersection multiplicity can be written as a cup product on the $K$-theory of perfect $A$-complexes which is the same as usual $K_0(A)$ when $A$ is regular. 

In \cite{GS1} Gillet and Soul\'e conclude with a similar theory of Chow groups with supports of regular rings (more general complete intersection rings). That is on a regular scheme $X$ with closed subvarieties $Y,Z$ there exists a pairing

\[ CH_Y^p(X) \otimes CH_Z^q(X) \to CH_{Y \cap Z}^{p+q}(X)_{\mathbb{Q}} \] 

\vspace{0.2cm}

The $\bigoplus_Y CH_Y^*(X)_{\mathbb{Q}}$ is a ring with unit $[X]$. Their strategy is use a $K$-theory with support of $\gamma$-filtration. However the intersection theory with support  obliges to extend the coefficients to $\mathbb{Q}$. In fact, they establish an isomorphism 

\[ CH_Y^p(X) \otimes \mathbb{Z}[1/(d-1)!] \cong Gr^p K_0^Y(X) \otimes \mathbb{Z}[1/(d-1)!] \]
 
\[ \ \Rightarrow CH_Y^p(X)_{\mathbb{Q}} \cong Gr^pK_0^Y(X)_{\mathbb{Q}} \] 

\vspace{0.2cm}

\noindent
Using reduction to the diagonal one shows that the intersection product agrees with the previously defined product, \cite{GS1}.

\vspace{0.2cm}

\noindent
\textbf{Conjecture:} \cite{G} On a regular noetherian scheme $X$ we have

\vspace{0.2cm}

\begin{center}
$F^i_{\text{cod}}(K_0^Y(X))*F^j_{\text{cod}}(K_0^Z(X)) \subset F^{i+j}_{\text{cod}}(K_0^{Y \cap Z}(X))$
\end{center}

\vspace{0.2cm}
 
The conjecture implies Serre vanishing conjecture. Suppose the supports $Y, Z$ of the two modules $M,N$ intersect in the close point $x \in X=\text{Spec}(A)$. If $[M] \in F_{\text{cod}}^pK_0^Y(X)$ and $[N] \in F_{\text{cod}}^qK_0^Z(X)$, with $p+q >n=\dim(A)$, then 

\vspace{0.2cm}

\begin{center}
$\chi(M,N)=[M] \cup [N] \in F_{\text{cod}}^{p+q} K_0^x(X) \subset F^{n+1}=0$
\end{center}

\vspace{0.2cm}

With in the time a significant proof of non-negativity was given by O. Gabber \cite{RO2}. His proof involves both algebraic and geometric insights concerning multiplicity as an Euler characteristic. The proof of O. Gabber, \cite{RO2} is based on a theorem of de Jang. The theorem states that for any regular (affine) scheme $V$ there exists a projective $\phi:X \to V$ which is of finite type over $V$. Let $Y^{\prime},Z^{\prime}$ be the subvarieties associated to $Y,Z$ via the de Jang map (called alterations of $ Y, Z$). The strategy is to compare $\chi(\mathcal{O}_{Y^{\prime}},\mathcal{O}_{Z^{\prime}})$ with $\chi(\mathcal{O}_{Y},\mathcal{O}_{Z})$. Gabber shows that 

\vspace{0.2cm}

\begin{center}
$\chi(Y',Z')= \chi(Y,\phi_*(F_{\bullet}))=\chi(Y,(\mathcal{O}_Z)^n)=n. \chi(Y,Z)$
\end{center}

\vspace{0.2cm}

\noindent
Since $\phi_*(\mathcal{O}_{Y'}) \cong (\mathcal{O}_{Y'})^n$. In another step he replaces the rings by $Gr_I A'$ and $Gr_IB$ where $B= A' \otimes_A k[Y]$. Set $E=Proj(Gr_I A'), \ M=Proj(Gr_I B)$. He shows 

\vspace{0.2cm}

\begin{center}
$\chi_E(M,Z')=\chi(Y',Z')$
\end{center}

\vspace{0.2cm}

\noindent
Set $E_s=Proj(Gr_I A' \otimes_A k)$ and $Z'_s=Proj(A'/I \otimes_A k)$. The positivity conjecture holds for $Y, Z$ if and only if

\vspace{0.2cm}
  
\begin{center}
$\chi_{E_s}(Proj(gr_IB \otimes k),Z'_s)  \geq 0$.
\end{center}

\vspace{0.2cm}

We have the following result of P. Roberts.

\vspace{0.2cm}

\begin{theorem} \cite{RO3}
If $A$ is a graded regular ring and $M,N$ graded finite modules with $l(M \otimes N) <\infty $, then the Serre multiplicity conjecture holds for $\chi(M,N)$, That is $\chi(M,N) \geq 0$ and is $\chi(M,N)>0$ if and only if $\dim(M)+\dim(N) =\dim(A)$.
\end{theorem}

The proof does not recognize any thing on vanishing or positivity separately. It only states the non-negativity. Serre multiplicity conjecture is known for graded regular rings.

\vspace{0.5cm}

\section{Intersection on singular schemes}

\vspace{0.5cm}

For all schemes ( quasi-projective or singular ) there exists a functorial homomorphism 

\begin{equation}
\tau:=\tau_X:K_0(X) \stackrel{\cong}{\longrightarrow} CH_*(X)_{\mathbb{Q}}
\end{equation}

\vspace{0.2cm}

\noindent
called Riemann-Roch homomorphism or (Todd genus), where $K_0(X)$ is the Grothendieck group of finite modules up to short exact sequences (sometimes denoted $G_0$ or $K_0'$). The strategy to extend the definition of $\tau$ to singular category, proceeds by embedding $X \hookrightarrow M$ in a smooth scheme $M$ and then define 

\[ \tau(E^{\bullet})=ch_M^X(E^{\bullet}). Td(M) , \qquad ch_M^X(E) \in H^*(M,M \setminus X) \]

\vspace{0.2cm}

\noindent
The definition of $ch_M^X(E^{\bullet})$ uses McPherson graph construction followed by the above mentioned embedding, \cite{I}, \cite{F} Chap. 18. The extension of $\tau$ to the singular category is unique. Therefore, it is an isomorphism for all schemes when tensor with $\mathbb{Q}$. 

Let $E^{\bullet}$ be any complex of vector bundles on a scheme $X$ which is exact off a closed subscheme $Y$. Then for any coherent sheaf $F$ on $X$, one has the following Riemann-Roch formula

\begin{equation}
\sum (-1)^i\tau_Y[H^i(E^{\bullet} \otimes F)]=Ch_Y^X(E^{\bullet}) \cap \tau_X(F)
\end{equation}

\vspace{0.2cm}

Formula (8) can be thought as a generalization to Tor formula as follows. It also recovers the fact that Serre definition of intersection multiplicity is in fact an identity of algebraic cycles. 

\vspace{0.3cm}

\noindent
\textbf{A proof of the vanishing conjecture:} In case of a regular embedding $f:Y \hookrightarrow X$ of codimension $d$ and normal bundle $N$, the above formula provides

\begin{equation}
\sum (-1)^i\tau_Y[Tor^i(\mathcal{O}_X, F)]=td(N)^{-1} \cap f^*\tau_X(F)
\end{equation}

\vspace{0.2cm}

\noindent
It follows that when $\dim(Supp F)=n$ one has

\begin{equation}
\sum (-1)^i Z_{n-d}[Tor^i(\mathcal{O}_X, F)]=td(N)^{-1} \cap f^* Z_n(F)
\end{equation}

\vspace{0.2cm}

\noindent
When $X$ is regular and $V,\ W$ closed subsets, applying the above formula to the diagonal embedding $X \hookrightarrow X \times_X X$ and $F=\mathcal{O}_{V \times W}$ gives the formula

\vspace{0.2cm}

\begin{center}
$[V].[W]=\sum (-1)^iZ_m[Tor^i(\mathcal{O}_V, \mathcal{O}_W)], \qquad m=\dim(V) +\dim(W) -\dim(V \cap W)$
\end{center}

\vspace{0.2cm}

\noindent
This formula proves the vanishing conjecture in non-proper intersection.

\vspace{0.2cm}

The Serre intersection multiplicity can be negative on non-regular rings. Hochster-Dutta-McLaughlin \cite{DHM},  give an example of two modules $M,N$ over 

\vspace{0.1cm}

\begin{center}
$A=k[x,y,u,v]_{\mathfrak{m}}/(xy-uv)_{\mathfrak{m}}$ 
\end{center}

\vspace{0.1cm}

\noindent
with $\chi(M,N)=-1$. We illustrate the following similar example generalized by Levine, \cite{L}. Set 

\vspace{0.1cm}

\begin{center}
$A=k[x,y,z,u,v,w]/(ux+vy+wz)$ 
\end{center}

\vspace{0.1cm}

\noindent
localized at the maximal ideal $\mathfrak{m}=(x,y,z,u,v,w)$. Let $\mathfrak{p}=(u,v,w)$. We wish to construct an $A$-module $N$ such that $\chi(N,A/\mathfrak{p})=-2$. The module $A/\mathfrak{p}$ has a minimal free resolution 

\begin{equation}
... \stackrel{\phi_3}{\rightarrow} A^4 \stackrel{\phi_4}{\rightarrow} A^4 \stackrel{\phi_3}{\rightarrow} A^4 \stackrel{\phi_2}{\rightarrow} A^3 \stackrel{\phi_1}{\rightarrow} A \to 0
\end{equation}

\vspace{0.2cm}

\noindent
where 

\vspace{0.2cm}

\begin{center}
$\phi_1= (u \ v \ w) \qquad \phi_2=\left( 
\begin{array}{cccc}
x  &  0  &  -w  &  v\\
y  &  w  &  0   &  -u\\
z  &  -v &  u   &  0
\end{array} \right)$\\[0.3cm]
$\phi_3=\left( 
\begin{array}{cccc}
0  &  u  &  v  &  w\\
u  &  0  &  z   & -y\\
v  &  -z &  0   &  x\\
w  &  y  &  -x  &  0
\end{array} \right) \qquad \phi_4=\left( 
\begin{array}{cccc}
0  &  x  &  y  &  z\\
x  &  0  &  -w   &  v\\
y  &  w &  0   &  -u\\
z  & -v  &  u  &  0
\end{array} \right)$

\end{center}

\vspace{0.2cm}

\noindent
If we assume $l(N)=55$ after tensoring this resolution with $N$, we get 

\begin{equation}
\chi(N,A/\mathfrak{p})=55-165+220-rank(\phi_4 \otimes N)=110-rank(\phi_4 \otimes N)
\end{equation}

\vspace{0.2cm}

\noindent
In the construction in \cite{D}, $rank(\phi_4 \otimes N)=112$. See \cite{DHM}, \cite{L} and \cite{RO3} for other examples.

\end{document}